\documentclass[leqno]{amsart}
\usepackage{amsfonts,amssymb,amsmath,amsgen,amsthm}
\usepackage{hyperref,color}
\usepackage{pdfsync}

\theoremstyle{plain}
\newtheorem{theorem}{Theorem}[section]
\newtheorem{definition}[theorem]{Definition}

\newtheorem{lemma}[theorem]{Lemma}

\newtheorem{proposition}[theorem]{Proposition}

\theoremstyle{remark}
\newtheorem{remark}[theorem]{Remark}

\def\C{{\mathbf C}}
\def\R{{\mathbf R}}

\def\({\left(}
\def\){\right)}
\def\<{\left\langle}
\def\>{\right\rangle}

\def\1{{\mathbf 1}}

\def\d{{\partial}}

\DeclareMathOperator{\RE}{Re}
\DeclareMathOperator{\IM}{Im}
\DeclareMathOperator{\diver}{div}

\DeclareMathOperator{\sgn}{sgn}

\numberwithin{equation}{section}

\date\today

\title{A model of Synchronization over Quantum Networks} 

\author{Paolo Antonelli}
\address{Gran Sasso Science Institute, L'Aquila}
\email{paolo.antonelli@gssi.it}

\author{Pierangelo Marcati}
\address{Gran Sasso Science Institute, L'Aquila \and Universit\`a dell'Aquila}
\email{marcati@univaq.it}

\begin{document}

\begin{abstract}
We investigate  a non-Abelian generalization of the Kuramoto model proposed by Lohe and given by $N$ quantum oscillators (``nodes") connected by a quantum network where the wavefunction at each node is distributed over quantum channels to all other connected nodes.  It leads to a system of Schr\"odinger equations coupled by nonlinear self-interacting potentials given by their correlations. We give a complete picture of synchronization results, given on the relative size of the natural frequency and the coupling constant, for two non-identical oscillators and show complete phase synchronization for arbitrary $N>2$ identical oscillators.  Our results are mainly based on the analysis of the ODE system satisfied by the correlations and on the introduction of a quantum order parameter, which is analogous to the one defined by Kuramoto in the classical model.  As a consequence of the previous results, we obtain the synchronization of the probability and the current densities defined via the Madelung transformations.
  
\end{abstract}
\maketitle 
\section{Introduction}
Synchronization is one of the most important phenomena in classical physics and its history can be dated back to the observation of two pendulum clocks by Huygens in 1665, that two pendulum clocks fastened to the same beam will synchronize (anti-phase). Examples of spontaneous synchronization are rhythmic applause in a large audience, heart beat (due to synchronization of 1000’s of cells), synchronous flashing of fireflies, \cite{SS, Str, PRK}.

The phenomenon of synchronization can be placed at the interface between statistical physics and nonlinear dynamics. We can consider a group of self-sustained oscillators, each with a (possibly random) proper frequency, we can observe that in certain regimes the interaction among the oscillators overcomes the frequency chaotic (random) behavior and yields them to synchronize and oscillate with the same frequency.   It is a key concept to the understanding of self-organization phenomena occurring in the fields of coupled oscillators of the dissipative type. Wave phenomena and pattern formation, may be viewed as typical synchronization phenomena in distributed systems while in contrast to the turbulence in reaction-diffusion systems, which is caused by desynchronization among local oscillators. Phase-transition-like phenomena can be characterized by the appearance or disappearance of collective oscillations.  Examples of multi-oscillator systems in living organisms and detailed accounts may be found in the Winfree's book (1980) \cite{Win}.  

The 1975 Kuramoto model \cite{Kur2} consists of N coupled phase oscillators with random proper  frequencies  described by a given probability density. They exhibit a non-equilibrium phase transition to a synchronized state, with the degree of synchronicity described by an order parameter. A one-dimensional Josephson array as a natural realization of the Kuramoto model was given by  \cite{WCS}.

Recently, the question if synchronization exists in quantum systems has attracted a lot of interest. There have been important attempts to address this problem theoretically, from communities as diverse as trapped atomic ensembles, Josephson junctions, nanomechanical systems, quantum cryptography, etc…..  

Quantum mechanics introduces two effects. The first is quantum noise, which is due to the oscillator gaining or losing individual phonons. The second effect is that the oscillators can be quantum mechanically entangled with each other. The general question is whether synchronization survives in the quantum limit, and how quantum mechanics qualitatively changes the behavior.   

Lohe \cite{Loh1, Loh2} proposed a simple model which was a non-Abelian generalization of the Kuramoto model. This model can be defined on any complex network where the variable at each node is an element of the unitary group, where a network of quantum oscillators in which quantum states are distributed among connected nodes by means of unitary transformations.  The complex system is given by $N$ quantum oscillators (``nodes") connected by a quantum network where the wavefunction at each node is distributed over quantum channels to all other connected nodes,  by for instance a quantum teleportation,  that infer on the evolution of the local wavefunction through a nonlinear interaction at a certain coupling strength.  In the Lohe paper \cite{Loh1}  the distribution of wavefunctions  is performed by constant unitary transformations that could in principle be implemented through fixed quantum circuits; for instance via optical interactions of single photons and atoms, that permit the distribution of entanglement across the network. 

Following the model in \cite{Loh1}, we consider a set of $N$ wave functions $\psi_1, \dotsc, \psi_N$ with $\psi_j:\R\times\R^d\to\C$, whose dynamics is described by the following system
\begin{equation*}
i\d_t\psi_j=-\frac12\Delta\psi_j+V_j\psi_j
+i\frac{K}{2N}\sum_{\ell=1}^N\left(\psi_\ell-\frac{\langle\psi_\ell, \psi_j\rangle}{\|\psi_j\|_{L^2}^2}\psi_j\right),
\qquad j=1, \dotsc, N,
\end{equation*}
Here the potentials $V_j$ are given by $V_j=V+\Omega_j$, where $V$ is an external potential acting on the whole set of wave functions and $\Omega_j\in\R$ are some given constants, playing the same role as the natural frequencies for the classical Kuramoto model.

In this paper we are going to study the asymptotic behavior of solutions to the above system and infer synchronization properties, see Definition \ref{def:sync} for a more precise statement of such properties.

As a first result we focus on the case of two non-identical oscillators. In this setup we give a complete picture of the asymptotic behavior for large times, depending on the relative size between the natural frequency and the coupling constant, see Theorem \ref{prop:two}. Then we study the general case $N\geq2$ with identical oscillators, i.e. $\Omega_j\equiv0$ for any $j=1, \dotsc, N$. 
Furthermore, by exploiting the results on the correlation functions, it is also possible to show synchronization in $H^1$, i.e. in the space of energy. This allows to infer the physically relevant result concerning the synchronization of the hydrodynamical quantities associated to the wave functions. We remark that this is formally equivalent to show the alignment of the velocity fields associated to each $\psi_j$.

To prove our results we are going to use two main ideas, the former is a finite dimensional reduction of the problem, which considers the auxiliary ODE system governing the dynamics of the correlations between the wave functions. The latter idea is to define an order parameter reminiscent of the analysis for classical Kuramoto model  and to study its asymptotic behavior in order to infer the synchronization properties for the whole Schr\"odinger-Lohe system.
To our knowledge our approach is completely new in the context of quantum synchronization, see also \cite{Pad}.

This paper is structured as follows: in Section 2 we state some preliminary results and we introduce the order parameter \eqref{eq:ord_par}. In Section 3 we determine the finite dimensional reduction, by writing down the ODE system satisfied by the correlation functions associated to the Schr\"odinger-Lohe system. In Section 4 we give our frequency synchronization result for the case $N=2$ and in Section 5 we show the phase synchronization in the case of identical oscillators. Finally in Section 6 we briefly comment on the alignment of the momenta associated to the wave functions.

\section{Preliminaries}
As already said, we consider the following Schr\"odinger-Lohe model
\begin{equation}\label{eq:SL1}
i\d_t\psi_j=-\frac12\Delta\psi_j+V_j\psi_j
+i\frac{K}{2N}\sum_{\ell=1}^N\left(\psi_\ell-\frac{\langle\psi_\ell, \psi_j\rangle}{\|\psi_j\|_{L^2}^2}\psi_j\right).
\qquad j=1, \dotsc, N,
\end{equation}
Here $\langle f, g\rangle:=\int_{\R^d}\bar f(x)g(x)\,dx$, and the potentials satisfy $V_j(x)=V+\Omega_j$, with $\Omega_j\in\R$, $K>0$.
In order to simplify the exposition in this note we will consider only space dependent potentials, namely $V\in L^{p}(\R^d)+L^\infty(\R^d)$, where $p>\max\{1, d/2\}$, however the same results stated here hold true also if we consider time-dependent potentials. It is possible to consider either integrable potentials, which can be treated as a perturbation of the Laplacian in the sense of Kato (see \cite{NS} for a quite general result) and potentials growing at most quadratically at spatial infinity, which cannot be treated as a perturbation of the Laplacian and require a more careful analysis, see for example \cite{Fuj, ADM}. 
Standard arguments from the theory of nonlinear Schr\"odinger equations, for instance see \cite{Caz}, yield to the global well-posedness of the system \eqref{eq:SL1} in $L^2(\R^d)$ and $H^1(\R^d)$.
\begin{proposition}\label{prop:lwp}
Let $\psi_{j, 0}\in L^2(\R^d)$ for any $j=1, \dotsc, N$, then there exists a unique global solution $(\psi_1, \dotsc, \psi_N)\in\mathcal C(\R_+; L^2(\R^d))$ to the system \eqref{eq:SL1}, with initial data 
$(\psi_1(0), \dotsc, \psi_N(0))=(\psi_{1, 0}, \dotsc, \psi_{N, 0})$. Furthermore the total mass of each individual wave function is conserved, i.e.
\begin{equation*}
\|\psi_j(t)\|_{L^2}=\|\psi_{j, 0}\|_{L^2},\qquad\forall\;t>0, j=1, \dotsc, N.
\end{equation*}
Moreover, if we assume $(\psi_{1, 0}, \dotsc, \psi_{N, 0})\in H^1(\R^d)$, then the corresponding global solution satisfies $(\psi_1, \dotsc, \psi_N)\in\mathcal C(\R_+;H^1(\R^d))$.
\end{proposition}
\begin{proof}
We omit here the proof of local well-posedness which is standard, see for instance the monograph \cite{Caz}. We only show the conservation of total mass for each wave function $\psi_j$. By differentiating the square of the $L^2-$norm with respect to time and by using equation \eqref{eq:SL1} we obtain
\begin{equation*}
\begin{aligned}
\frac{d}{dt}\left(\frac12\int|\psi_j(t, x)|^2\,dx\right)=&\int\RE\bigg\{\bar\psi_j\Big(\frac{i}{2}\Delta\psi_j-iV_j\psi_j
+\frac{K}{2N}\sum_{\ell=1}^N\left(\psi_\ell-\frac{\langle\psi_\ell, \psi_j\rangle}{\|\psi_j\|_{L^2}^2}\psi_j\right)\Big)\bigg\}\,dx\\
=&\frac{K}{2N}\sum_{\ell=1}^N\RE\left(\langle\psi_j, \psi_\ell\rangle-\frac{\langle\psi_\ell, \psi_j\rangle}{\|\psi_j\|_{L^2}^2}\|\psi_j\|_{L^2}^2\right)=0.
\end{aligned}
\end{equation*}
The conservation of the total masses, combined with the usual blow-up alternative, then yields the existence of global in time solutions to \eqref{eq:SL1}.
\newline
The existence of the global $H^1$ solution emanated from $H^1$ initial data follows from the global well-posedness in $L^2$ and the persistence of regularity property for nonlinear Schr\"odinger equations.
\end{proof}
Since the total mass of each wave function is conserved, then without loss of generality we may assume $\|\psi_j(0)\|_{L^2}=1$, for any $j=1, \dotsc, N$. Consequently the system \eqref{eq:SL1} can be written in the following way
\begin{equation}\label{eq:SL}
i\d_t\psi_j=-\frac12\Delta\psi_j+V_j\psi_j+i\frac{K}{2N}\sum_{\ell=1}^N\left(\psi_\ell-\langle\psi_\ell, \psi_j\rangle\psi_j\right).
\end{equation}
Furthermore, without loss of generality we may also assume that $\sum_j\Omega_j=0$. Indeed if we have
\begin{equation*}
\frac1N\sum_j\Omega_j=\alpha\neq0,
\end{equation*}
we can define
\begin{equation*}
\psi'_j(t, x)=e^{-i\alpha t}\psi_j(t, x),\qquad\forall\;j=1, \dotsc, N.
\end{equation*}
and we can see that the transformed wave functions satisfy the following Schr\"odinger-Lohe system
\begin{equation*}
i\d_t\psi'_j=-\frac12\Delta\psi'_j+\tilde V_j\psi'_j+i\frac{K}{2N}\sum_{\ell=1}^N\left(\psi'_\ell-\langle\psi'_\ell, \psi'_j\rangle\psi'_j\right),
\end{equation*}
where $\tilde V_j=V+\tilde\Omega_j$, $\tilde\Omega_j=\Omega_j-\alpha$ and $\frac1N\sum_j\tilde\Omega_j=\frac1N\sum_j\Omega_j-\alpha=0$. Hence from now on we will assume 
\begin{equation}\label{eq:renorm}
\sum_j\Omega_j=0.
\end{equation}
In order to show the synchronization properties for the system \eqref{eq:SL} we introduce, in analogy with the classical Kuramoto model, the following order parameter
\begin{equation}\label{eq:ord_par}
\zeta:=\frac1N\sum_{\ell=1}^N\psi_\ell.
\end{equation}
The concept of complex order parameter was introduced by Kuramoto (see for instance \cite{Kur}, chapt. 5 ) in analogy to the analysis of thermodynamic phase transitions. 
By using the definition in \eqref{eq:ord_par} it is then possible to write the equations in \eqref{eq:SL} in the following way
\begin{equation}\label{eq:SL2}
i\d_t\psi_j=-\frac12\Delta\psi_j+V_j\psi_j+i\frac{K}{2}\left(\zeta-\langle\zeta, \psi_j\rangle\psi_j\right).
\end{equation}
Furthermore, by averaging the equations in \eqref{eq:SL2} we also write the governing dynamics for the order parameter,
\begin{equation}\label{eq:op}
i\d_t\zeta=-\frac12\Delta\zeta+V\zeta+\frac1N\sum_{\ell=1}^N\Omega_\ell\psi_\ell+i\frac{K}{2}\left(\zeta-\frac1N\sum_{\ell=1}^N\langle\zeta, \psi_\ell\rangle\psi_\ell\right).
\end{equation}

\begin{definition}\label{def:sync}
We say that the system \eqref{eq:SL1} exhibits \emph{complete frequency synchronization} if
\begin{equation*}
\lim_{t\to\infty}\|\psi_j(t)-\psi_k(t)\|_{L^2}=c_{jk},\quad\forall\;j, k=1, \dotsc, N
\end{equation*}
and 
\begin{equation*}
\lim_{t\to\infty}\|\zeta(t)\|_{L^2}=c\in(0, 1).
\end{equation*}
We say that the system \eqref{eq:SL1} exhibits \emph{complete phase synchronization} if
\begin{equation*}
\lim_{t\to\infty}\|\psi_j(t)-\psi_k(t)\|_{L^2}=0,\quad\forall\;j, k=1, \dotsc, N,
\end{equation*}
and 
\begin{equation*}
\lim_{t\to\infty}\|\zeta(t)\|_{L^2}=1.
\end{equation*}
\end{definition}
\begin{remark}
From the definition of the order parameter $\zeta$, it is straightforward to see that $\|\zeta\|_{L^2}=1$, if and only if $\|\psi_j-\psi_k\|_{L^2}=0$ for any $j, k=1, \dotsc, N$.
Indeed we have
\begin{equation*}
1-\|\zeta\|_{L^2}^2=\frac{1}{2N^2}\sum_{j, k=1}^N\|\psi_j-\psi_k\|_{L^2}^2.
\end{equation*}
\end{remark}
\section{The system of ODEs for the correlation functions}
To infer the synchronization properties for the Schr\"odinger-Lohe model \eqref{eq:SL} we look at the correlations between the wave functions in the system; for this purpose we define
\begin{equation*}
r_{jk}:=\RE\langle\psi_j, \psi_k\rangle, \quad s_{jk}:=\IM\langle\psi_j, \psi_k\rangle.
\end{equation*}
By the definition it follows $r_{jk}=r_{kj}, s_{jk}=-s_{kj}$, for any $j, k=1, \dotsc, N$. 

By using the system \eqref{eq:SL}, it is straightforward to deduce the system of ODEs describing the coupled dynamics for $r_{jk}, s_{jk}$, namely
\begin{equation}\label{eq:ODE_full}
\left\{\begin{aligned}
\frac{d}{dt}r_{jk}=&-(\Omega_j-\Omega_k)s_{jk}+\frac{K}{2N}\sum_{\ell=1}^N\left[(r_{j\ell}+r_{\ell k})(1-r_{jk})+(s_{j\ell}+s_{\ell k})s_{jk}\right]\\
\frac{d}{dt}s_{jk}=&(\Omega_j-\Omega_k)r_{jk}+\frac{K}{2N}\sum_{\ell=1}^N\left[-(r_{j\ell}+r_{\ell k})s_{jk}+(s_{j\ell}+s_{\ell k})(1-r_{jk})\right].
\end{aligned}\right.
\end{equation}
The derivation of \eqref{eq:ODE_full} is similar to what will be done later in Proposition \ref{prop:macro_corr}, hence we omit the details for \eqref{eq:ODE_full} and we address the reader to Remark 
\ref{rmk:ODE_full}. The above system becomes very simple in the case of two oscillators; indeed for $N=2$ we have
\begin{equation}\label{eq:ODE_two}
\left\{\begin{aligned}
\frac{d}{dt}r_{12}=&-2\Omega s_{12}+\frac{K}{2}(1-r_{12}^2+s_{12}^2)\\
\frac{d}{dt}s_{12}=&2\Omega r_{12}-Kr_{12}s_{12},
\end{aligned}\right.
\end{equation}
which can be studied directly, see Section \ref{sect:two} for more details.

However, for $N\geq3$ the system becomes more complicated and it will be more convenient to deal with the correlations of the wave functions with the order parameter.
For this purpose we also define
\begin{equation}\label{eq:macro_corr}
\tilde r_j:=\RE\langle\zeta, \psi_j\rangle=\frac1N\sum_{\ell=1}^Nr_{\ell j},\qquad\tilde s_j:=\IM\langle\zeta, \psi_j\rangle=\frac1N\sum_{\ell=1}^Ns_{\ell j}.
\end{equation}
These functions can be interpreted as macroscopic quantities measuring the correlation between the j-th wave function $\psi_j$ and the order parameter $\zeta$.
This definition readily implies that
\begin{equation*}
\frac{1}{N}\sum_{j=1}^N\tilde r_j=\|\zeta\|_{L^2}^2,\qquad\frac1N\sum_{j=1}^N\tilde s_j=0.
\end{equation*}
\begin{remark}
By using \eqref{eq:op} we can infer the ODE satisfied by the $L^2$ norm of the order parameter, namely
\begin{equation*}
\frac{d}{dt}\|\zeta(t)\|_{L^2}^2=\frac{2}{N}\sum_{\ell=1}^N\Omega_\ell\tilde s_\ell+K\left(\|\zeta(t)\|_{L^2}^2-\frac1N\sum_{\ell=1}^N(\tilde r_\ell^2-\tilde s_\ell^2)\right).
\end{equation*}
However it seems at the moment that no much information about complete synchronization can be obtained from the previous ODE. This fact is in strong contrast with the classical Kuramoto model, see for example \cite{BCM}, where the behavior of the modulus of the order parameter allows to infer the synchronization property for the whole system.
\end{remark}
In order to show the complete synchronization we are going to study the dynamics of the macroscopic correlations $\tilde r_j, \tilde s_j$. As it will be clear from the statement of next Proposition, those quantities do not satisfy a closed system of ODEs, nevertheless the analysis of this system will be sufficient to deduce the complete synchronization property for the Schr\"odinger-Lohe model \eqref{eq:SL}.
\begin{proposition}\label{prop:macro_corr}
Let $\psi_j$ solve the Schr\"odinger-Lohe system \eqref{eq:SL}, then the macroscopic correlations satisfy
\begin{equation}\label{eq:ODE_corr}
\left\{\begin{aligned}
\frac{d}{dt}\tilde r_j=&\Omega_j\tilde s_j-\frac1N\sum_{\ell=1}^N\Omega_\ell s_{\ell j}+\frac{K}{2}\left[\tilde r_j-\tilde r_j^2+\tilde s_j^2+\frac1N\sum_{\ell=1}^N(\tilde r_\ell-\tilde r_\ell r_{\ell j}-\tilde s_\ell s_{\ell j})\right]\\
\frac{d}{dt}\tilde s_j=&-\Omega_j\tilde r_j+\frac1N\sum_{\ell=1}^N\Omega_\ell r_{\ell j}+\frac{K}{2}\left[\tilde s_j-2\tilde r_j\tilde s_j+\frac1N\sum_{\ell=1}^N(r_{\ell j}\tilde s_\ell-s_{\ell j}\tilde r_\ell)\right].
\end{aligned}\right.
\end{equation}
\end{proposition}
\begin{proof}
By considering the equations for $\psi_j$ \eqref{eq:SL2} and $\zeta$ \eqref{eq:op} we may write
\begin{equation*}
\begin{aligned}
\d_t(\bar\zeta\psi_j)=&\left(-\frac{i}{2}\Delta\bar\zeta+iV\bar\zeta+\frac{i}{N}\sum_{\ell=1}^N\Omega_\ell\bar\psi_\ell+\frac{K}{2}\left(\bar\zeta-\frac1N\sum_{\ell=1}^N\langle\psi_\ell, \zeta\rangle\bar\psi_\ell\right)\right)\psi_j\\
&+\bar\zeta\left(\frac{i}{2}\Delta\psi_j-iV_j\psi_j+\frac{K}{2}\left(\zeta-\langle\zeta,\psi_j\rangle\psi_j\right)\right)\\
=&-\frac{i}{2}\left(\psi_j\Delta\bar\zeta-\bar\zeta\Delta\psi_j\right)-i\Omega_j\bar\zeta\psi_j+\frac{i}{N}\sum_{\ell=1}^N\Omega_\ell\bar\psi_\ell\psi_j\\
&+\frac{K}{2}\left(\left(1-\langle\zeta, \psi_j\rangle\right)\bar\zeta\psi_j+|\zeta|^2-\frac1N\sum_{\ell=1}^N\langle\psi_\ell, \zeta\rangle\bar\psi_\ell\psi_j\right).
\end{aligned}
\end{equation*}
We notice that the first term on the right hand side is the divergence of a current density,
\begin{equation*}
-\frac{i}{2}\left(\psi_j\Delta\bar\zeta-\bar\zeta\Delta\psi_j\right)=-\diver\left(\frac{i}{2}\left(\psi_j\nabla\bar\zeta-\bar\zeta\nabla\psi_j\right)\right),
\end{equation*}
so that by integrating the above expression on the whole space we obtain
\begin{equation}\label{eq:corr}
\begin{aligned}
\frac{d}{dt}\langle\zeta, \psi_j\rangle=&-i\Omega_j\langle\zeta, \psi_j\rangle+\frac{i}{N}\sum_{\ell=1}^N\Omega_\ell\langle\psi_\ell, \psi_j\rangle\\
&+\frac{K}{2}\left(\left(1-\langle\zeta, \psi_j\rangle\right)\langle\zeta, \psi_j\rangle+\frac1N\sum_{\ell=1}^N\langle\psi_\ell, \zeta\rangle(1-\langle\psi_\ell, \psi_j\rangle)\right),
\end{aligned}
\end{equation}
where we also used the fact that $\|\zeta\|_{L^2}^2=\frac1N\sum_{\ell}\langle\psi_\ell, \zeta\rangle$. By using the definitions of the correlation functions above and by separating the real and imaginary parts in \eqref{eq:corr} we then get \eqref{eq:ODE_corr}.
\end{proof}
\begin{remark}\label{rmk:ODE_full}
Analogously to system \eqref{eq:ODE_corr}, it is possible to derive the system of ODEs \eqref{eq:ODE_full} proceeding in a similar way and by combining equations \eqref{eq:SL2}. We only need to consider the set of ODEs satisfied by the correlation functions
\begin{equation*}
z_{jk}(t):=\langle\psi_j, \psi_k\rangle(t)=r_{jk}(t)+is_{jk}(t).
\end{equation*}
\end{remark}
\section{$L^2$ and $H^1$ complete synchronization for the two oscillator case}\label{sect:two}
In this Section we focus on the case of two non-identical oscillators. namely $N=2$ and $V_1\neq V_2$. First of all let us notice that, because of the condition \eqref{eq:renorm}, we can choose $\Omega_1=-\Omega_2=:\Omega\geq0$, so that the Schr\"odinger-Lohe system \eqref{eq:SL} becomes
\begin{equation}\label{eq:SLtwo}
\left\{\begin{aligned}
i\d_t\psi_1=&-\frac12\Delta\psi_1+V\psi_1+\Omega\psi_1+i\frac{K}{4}\left(\psi_2-\langle\psi_2, \psi_1\rangle\psi_1\right)\\
i\d_t\psi_2=&-\frac12\Delta\psi_2+V\psi_2-\Omega\psi_2+i\frac{K}{4}\left(\psi_1-\langle\psi_1, \psi_2\rangle\psi_2\right).
\end{aligned}\right.
\end{equation}
In this case, without any further assumptions, we are going to give a complete picture of synchronization properties for system \eqref{eq:SLtwo}, both in the $L^2$ and $H^1$ setup.

The result in $H^1$ also yields the physically relevant property of synchronization for the two momenta (probability and current densities), defined via the Madelung transformation.
To our knowledge it is the first result in this direction.

Next Theorem shows in which regimes system \eqref{eq:SLtwo} exhibits synchronization and in which other regimes it is not possible to have any synchronization result. Furthermore we also give the optimal convergence rates. As we will show, this result depends on the parameter 
$\Lambda:=\frac{2\Omega}{K}\geq0$ measuring the relative size between the natural frequency and the coupling constant.
\begin{theorem}\label{prop:two}
Let $\psi_{1, 0}, \psi_{2, 0}\in L^2(\R^d)$. Then, depending on the relative size between the coupling constant and the natural frequency, we have the following scenario:
\begin{enumerate}
\item for $0\leq\Lambda<1$, if $\psi_{1, 0}, \psi_{2, 0}$ are such that
\begin{equation*}
\langle\psi_{1, 0}, \psi_{2, 0}\rangle\neq-\sqrt{1-\Lambda^2}+i\Lambda,
\end{equation*}
then we have complete synchronization,
\begin{equation}\label{eq:gdr_two}
\|e^{i\phi}\psi_1(t)-\psi_2(t)\|_{L^2}\lesssim e^{-\sqrt{K^2-4\Omega^2}t},\quad\textrm{as}\;t\to\infty
\end{equation}
and 
\begin{equation*}
\lim_{t\to\infty}\|\psi_1(t)-\psi_2(t)\|_{L^2}=\left|1-e^{i\phi}\right|,
\end{equation*}
where $\phi$ is defined by $\arcsin \Lambda$ = $\phi$
\item for $\Lambda=1$, if $\psi_{1, 0}, \psi_{2, 0}$ are such that
\begin{equation*}
\langle\psi_{1, 0}, \psi_{2, 0}\rangle\neq i,
\end{equation*}
then
\begin{equation*}
\|i\psi_1(t)-\psi_2(t)\|_{L^2}\lesssim t^{-1},\quad\textrm{as}\;t\to\infty
\end{equation*}
and 
\begin{equation*}
\lim_{t\to\infty}\|\psi_1(t)-\psi_2(t)\|_{L^2}=\sqrt2;
\end{equation*}
\item for $\Lambda>1$, the correlations are periodic in time.
\end{enumerate}
\end{theorem}
\begin{proof}
By using the equations in \eqref{eq:SLtwo} we can derive the ODE satisfied by the correlation between the two wave functions. If we define $z(t):=\langle\psi_1, \psi_2\rangle(t)$, then
\begin{equation}\label{eq:ODE_z}
\dot z=2i\Omega+\frac{K}{2}(1-z^2).
\end{equation}
Under the assumption $\Lambda<1$, the two stationary points $z_1=\sqrt{1-\Lambda^2}+i\Lambda=:e^{i\phi}$, $z_2=-\sqrt{1-\Lambda^2}+i\Lambda$, are symmetric with respect to the imaginary axis. By assuming $z(0)\neq z_{1, 2}$, then it is straightforward to see that the solution $z(t)$ to \eqref{eq:ODE_z} is given by
\begin{equation*}
z(t)=\frac{e^{i\phi}+e^{-i\phi}\frac{z(0)-e^{i\phi}}{z(0)+e^{-i\phi}}e^{-\sqrt{K^2-4\Omega^2}t}}{1-\frac{z(0)-e^{i\phi}}{z(0)+e^{i\phi}}e^{-\sqrt{K^2-4\Omega^2}t}},
\end{equation*}
so that $|z(t)-e^{i\phi}|\lesssim e^{-\sqrt{K^2-4\Omega^2}t}$. Thus, since $\|\psi_1\|_{L^2}=\|\psi_2\|_{L^2}=1$, we have
\begin{equation*}
\|e^{i\phi}\psi_1(t)-\psi_2(t)\|_{L^2}^2=2\left(1-\RE(e^{-i\phi}z(t))\right)\lesssim e^{-\sqrt{K^2-4\Omega^2}t},
\end{equation*}
and consequently
\begin{equation*}
\lim_{t\to\infty}\|\psi_1(t)-\psi_2(t)\|_{L^2}=|1-e^{i\phi}|.
\end{equation*}
On the other hand, let us assume $\Lambda=1$, then we have $z_1=z_2=i$ and in this case the solution to \eqref{eq:ODE_z} is given by
\begin{equation*}
z(t)=i+\left(\frac{K}{2}t+\frac{1}{z(0)-i}\right)^{-1}.
\end{equation*}
Finally, we see that for $\Lambda>1$ the solution is time periodic.
\end{proof}

\begin{remark}
Let us consider again the ODE in \eqref{eq:ODE_z}, if we formally write $z=re^{i\theta}$, then we see that the pair $(r, \theta)$ solves the following system
\begin{equation*}
\left\{\begin{aligned}
\dot r=&\frac{K}{2}\cos\theta(1-r^2)\\
\dot\theta=&2\Omega-K\sin\theta+\frac{K}{2}(1-r^2)\sin\theta,
\end{aligned}\right.
\end{equation*}
so when $r\equiv1$, we recover the classical Kuramoto model for two oscillators. It implies, in particular, that all solutions $z(t)=e^{i\theta(t)}$, where $\theta(t)$ is a solution of the classical Kuramoto model, are also solutions to \eqref{eq:ODE_z}.

On the other hand, if we write $z(t)=r_{12}(t)+is_{12}(t)$, $r_{12}(t)=\RE\langle\psi_1, \psi_2\rangle(t)$, $s_{12}(t)=\IM\langle\psi_1, \psi_2\rangle(t)$, then by separating the real and imaginary parts in \eqref{eq:ODE_z} we recover system \eqref{eq:ODE_two}.
\end{remark}
Next Proposition explictily gives the asymptotically synchronized states for the Schr\"odinger-Lohe system \eqref{eq:SLtwo}, when  $N=2$.
\begin{proposition}
Let $\psi_{1, 0}, \psi_{2, 0}\in L^2(\R^d)$ and let us assume that $0\leq\Lambda<1$, $\langle\psi_{1, 0}, \psi_{2, 0}\rangle\neq-\sqrt{1-\Lambda^2}+i\Lambda$. Then we have
\begin{equation*}
\lim_{t\to\infty}\|\psi_1(t)-e^{-i(-\frac12\Delta+V)t}\tilde\psi_1\|_{L^2}=\lim_{t\to\infty}\|\psi_2(t)-e^{-i(-\frac12\Delta+V)t}\tilde\psi_2\|_{L^2}=0,
\end{equation*}
where
\begin{equation*}
\begin{aligned}
\tilde\psi_1=&\psi_{1, 0}-i\int_0^\infty e^{i(-\frac12\Delta+V)s}\left(\Omega\psi_1+i\frac{K}{4}(\psi_2\langle\psi_2, \psi_1\rangle\psi_1)\right)(s)\,ds,\\
\tilde\psi_2=&\psi_{2, 0}-i\int_0^\infty e^{i(-\frac12\Delta+V)s}\left(-\Omega\psi_2+i\frac{K}{4}(\psi_1\langle\psi_1, \psi_2\rangle\psi_2)\right)(s)\,ds.
\end{aligned}
\end{equation*}
\end{proposition}
\begin{proof}
Under our assumptions, we know that
\begin{equation*}
\|e^{i\phi}\psi_1(t)-\psi_2(t)\|_{L^2}\lesssim e^{-\sqrt{K^2-4\Omega^2}t},
\end{equation*}
hence it is straightforward to see that
\begin{equation*}
\|\Omega\psi_1+i\frac{K}{4}(\psi_2-\langle\psi_2, \psi_1\rangle\psi_1)\|_{L^2}\lesssim e^{-\sqrt{K^2-4\Omega^2}t}.
\end{equation*}
Let us consider the integral equation for $\psi_1$, from \eqref{eq:SLtwo} we may write
\begin{equation*}
\psi_1(t)=e^{-i(-\frac12\Delta+V)t}\psi_{1, 0}-i\int_0^te^{-i(-\frac12\Delta+V)(t-s)}\left(\Omega\psi_1+i\frac{K}{4}(\psi_2-\langle\psi_2, \psi_1\rangle\psi_1)\right)(s)\,ds,
\end{equation*}
and consequently, 
\begin{equation*}
e^{i(-\frac12\Delta+V)t}\psi_1(t)=\psi_{1, 0}-i\int_0^te^{i(-\frac12\Delta+V)s}\left(\Omega\psi_1+i\frac{K}{4}(\psi_2-\langle\psi_2, \psi_1\rangle\psi_1)\right)(s)\,ds.
\end{equation*}
It is straightforward to infer that the time integral on the right hand side has a strong limit in $L^2$ as $t\to\infty$, indeed
\begin{equation*}
\left\|\int_0^te^{i(-\frac12\Delta+V)s}\left(\Omega\psi_1+i\frac{K}{4}(\psi_2-\langle\psi_2, \psi_1\rangle\psi_1)\right)(s)\,ds\right\|_{L^2}\lesssim\int_0^te^{-\sqrt{K^2-4\Omega^2}s}\,ds.
\end{equation*}
consequently, we may define
\begin{equation*}
\tilde\psi_1=\psi_{1, 0}-i\int_0^\infty e^{i(-\frac12\Delta+V)s}\left(\Omega\psi_1+i\frac{K}{4}(\psi_2-\langle\psi_2, \psi_1\rangle\psi_1)\right)(s)\,ds
\end{equation*}
and the result in Proposition \ref{prop:two} now implies that
\begin{equation*}
\|e^{i(-\frac12\Delta+V)t}\psi_1(t)-\tilde\psi_1\|_{L^2}\lesssim e^{-\sqrt{K^2-4\Omega^2}t}.
\end{equation*}
By the unitarity of $e^{-i(-\frac12\Delta+V)t}$ we then obtain the desired result. An analogous calculation shows the asymptotic behavior for $\psi_2(t)$.
\end{proof}
We turn now our attention to the $H^1$ setting. From Proposition \ref{prop:lwp} we know that the global solution emanated from $(\psi_{1, 0}, \psi_{2, 0})\in H^1(\R^d)$ lives also in $H^1$, namely 
$(\psi_1, \psi_2)\in\mathcal C(\R_+; H^1(\R^d))$. Our aim now is to show frequency synchronization in the $H^1$ setup. First of all we show that the total energy of the system
\begin{equation}\label{eq:en2}
\begin{aligned}
E(t)=&\int\frac12|\nabla\psi_1|^2+\frac12|\nabla\psi_2|^2+V(|\psi_1|^2+|\psi_2|^2)+\Omega(|\psi_1|^2-|\psi_2|^2)\,dx\\
=&\int\frac12|\nabla\psi_1|^2+\frac12|\nabla\psi_2|^2+V(|\psi_1|^2+|\psi_2|^2)\,dx.
\end{aligned}
\end{equation}
is uniformly bounded for all times. Notice that the last equality follows from $\|\psi_1(t)\|_{L^2}=\|\psi_2(t)\|_{L^2}=1$ and condition \eqref{eq:renorm}.
\begin{lemma}
There exists $C \geq 1$, such that for all $t \geq 0$, it follows  
\begin{equation*}
E(t)\leq CE(0).
\end{equation*}
\end{lemma}
\begin{proof}
Let us differentiate the energy in \eqref{eq:en2} with respect to time, by using system \eqref{eq:SLtwo} we have
\begin{equation*}
\begin{aligned}
\frac{d}{dt}E(t)=&\frac{K}{2}\int\RE\bigg\{\left(-\frac12\Delta\bar\psi_1+(V+\Omega)\bar\psi_1\right)(\psi_2-\langle\psi_2, \psi_1\rangle\psi_1)\\
&\quad+\left(-\frac12\Delta\bar\psi_2+(V-\Omega)\bar\psi_2\right)(\psi_1-\langle\psi_1, \psi_2\rangle\psi_2)\bigg\}\,dx\\
=&-\frac{K}{2}r_{12}(t)E(t)+K\int\RE\left\{\frac12\nabla\bar\psi_1\cdot\nabla\psi_2+V\bar\psi_1\psi_2\right\}\,dx,
\end{aligned}
\end{equation*}
where we recall $r_{12}=\RE\langle\psi_1, \psi_2\rangle$.
We write the last term on the right hand side as
\begin{equation*}
\int\RE\left\{\frac12\nabla\bar\psi_1\cdot\nabla\psi_2+V\bar\psi_1\psi_2\right\}\,dx=-\frac12\tilde E(t)+\frac12E(t),
\end{equation*}
where we defined the relative energy
\begin{equation*}
\tilde E(t)=\int\frac12|\nabla(\psi_1-\psi_2)|^2+V|\psi_1-\psi_2|^2\,dx.
\end{equation*}
By using the formula above we get
\begin{equation*}
\frac{d}{dt}E(t)=\frac{K}{2}(1-r_{12}(t))E(t)-\frac{K}{2}\tilde E(t)\leq\frac{K}{2}(1-r_{12}(t))E(t).
\end{equation*}
The lemma then follows from Gronwall's inequality and $\int_0^t(1-r_{12}(t))\,dt<\infty$.
\end{proof}
Next result shows frequency synchronization even in the $H^1$ framework.
\begin{proposition}\label{prop:H1_sync}
Let $(\psi_{1, 0}, \psi_{2, 0})\in H^1(\R^d)$ be such that either point (1) or (2) in Theorem \ref{prop:two} holds true, then it follows 
\begin{equation*}
\|e^{i\phi}\psi_1(t)-\psi_2(t)\|_{H^1}\lesssim e^{-\sqrt{K^2-4\Omega^2}t}, \quad\textrm{as}\;t\to\infty.
\end{equation*}
\end{proposition}
\begin{proof}
Having in mind the synchronization result stated in Proposition \ref{prop:two} we want to show that $\psi_d=e^{i\phi}\psi_1-\psi_2$ converges exponentially to zero for large times. By using system \eqref{eq:SLtwo} we deduce that the equation satisfied by $\psi_d$ is given by 
\begin{equation*}
\begin{aligned}
i\d_t\psi_d=-\frac12\Delta\psi_d+V\psi_d+i\frac{K}{4}(\RE(e^{-i\phi}+\langle\psi_1, \psi_2\rangle)(\psi_2-e^{i\phi}\psi_1) \\
+i\IM(e^{i\phi}+\langle\psi_1, \psi_2\rangle)(\psi_2+e^{i\phi}\psi_1).
\end{aligned}
\end{equation*}
To study the asymptotic behavior of $\psi_d$, we consider the following energy functional associated to the previous equation,
\begin{equation*}
E_d(t)=\int\frac12|\nabla\psi_d|^2+V|\psi_d|^2\,dx.
\end{equation*}
By differentiating with respect to time, we get
\begin{equation*}
\begin{aligned}
\frac{d}{dt}E_d(t)=&\frac{K}{2}\int\RE\bigg\{\left(-\frac12\Delta\bar\psi_d+V\bar\psi_d\right)\Big(\RE(e^{-i\phi}+\langle\psi_1, \psi_2\rangle)(\psi_2-e^{i\phi}\psi_1)\\
&\qquad+i\IM(e^{i\phi}+\langle\psi_1, \psi_2\rangle)(\psi_2+e^{i\phi}\psi_1\Big)\bigg\}\,dx\\
=&-\frac{K}{2}\left(\sqrt{1-\Lambda^2}+r_{12}(t)\right)E_d(t)\\
&-\frac{K}{2}\left(\Lambda-s_{12}(t)\right)\int\IM\left\{\left(-\frac12\Delta\bar\psi_d+V\bar\psi_d\right)(\psi_2+e^{i\phi}\psi_1\right\}\,dx.
\end{aligned}
\end{equation*}
the last integral on the right hand side equals
\begin{equation*}
2\int\IM\left\{e^{-i\phi}\left(\frac12\nabla\bar\psi_1\cdot\nabla\psi_2+V\bar\psi_1\psi_2\right)\right\}\,dx,
\end{equation*}
whose absolute value is bounded by the total energy $E(t)\leq CE(0)$. On the other hand, let us also recall that, under our assumptions, we have
\begin{equation*}
\left|\sqrt{1-\Lambda^2}-r_{12}(t)\right|+\left|\Lambda-s_{12}(t)\right|\lesssim e^{-\sqrt{K^2-4\Omega^2}t},\quad\textrm{as}\;t\to\infty,
\end{equation*}
hence  by the Gronwall's inequality, we obtain
\begin{equation*}
\begin{aligned}
E_d(t)\leq& e^{-\sqrt{K^2-4\Omega^2}t}E_d(0)+C\int_0^te^{-\sqrt{K^2-4\Omega^2}(t-s)}\left|\Lambda-s_{12}(s)\right|E(s)\,ds\\
\leq&e^{-\sqrt{K^2-4\Omega^2}t}\left(E_d(0)+CtE(0)\right).
\end{aligned}
\end{equation*}
\end{proof}

\section{Complete synchronization for $N$ oscillators}
In this Section we consider the general case with $N$ identical oscillators, then we can assume $\Omega_j\equiv0$ for any $j=1, \dotsc, N$.
Our first result is dedicated to characterize the asymptotic states of the Schr\"odinger-Lohe model.

\begin{proposition}
We have
\begin{equation}\label{eq:decoupl}
\zeta-\langle\zeta, \psi_j\rangle\psi_j=0,\qquad\forall\;j=1, \dotsc, N,
\end{equation}
if and only if one of the two cases hold:
\begin{itemize}
\item $\zeta=0$ (incoherent dynamics);
\item upon relabelling the wave functions, we have $\psi_1=\dotsc=\psi_k=-\psi_{k+1}=\dotsc=-\psi_N$, for some $k>\frac{N}{2}$ and consequently
\begin{equation*}
\zeta=\left(\frac{2k}{N}-1\right)\psi_1.
\end{equation*}
\end{itemize}
\end{proposition}
\begin{proof}
Recall the definition $\zeta=\frac1N\sum\psi_j$, clearly if $\zeta=0$ then \eqref{eq:decoupl} is satisfied for all $j=1, \dotsc, N$. Thus let us assume in the sequel that $\zeta\neq0$. For \eqref{eq:decoupl} to hold true, then we have $\zeta=c_j\psi_j$, where $c_j=\langle\zeta, \psi_j\rangle\in\mathbb C$. On the other hand, since $\|\psi_j\|_{L^2}=1$, then
\begin{equation*}
\zeta-\langle\zeta, \psi_j\rangle\psi_j=(c_j-\bar c_j)\psi_j=0,
\end{equation*}
and consequently $c_j\in\mathbb R$. Furthermore this also implies that the absolute value of the constant $c_j$ must be independent on $j$, as $\|\zeta\|_{L^2}=|c_j|=c$, for any $j=1, \dotsc, N$, with $c\in(0, 1]$. On the other hand we have
\begin{equation*}
\zeta=\frac1N\sum\psi_j=\frac1N\sum\frac{1}{c_j}\zeta
\end{equation*}
and consequently
\begin{equation*}
1=\frac1N\sum\frac{1}{c_j}=\frac{1}{Nc}\sum\sgn c_j,
\end{equation*}
where $\sgn c_j$ is the sign of the constant $c_j$. Let us assume that there are $k$ positive $c_j$; upon relabelling the wave functions we may assume that $c_1=\dotsc=c_k=-c_{k+1}=\dotsc=c_N\geq0$. The formula above then implies that $2k-N=Nc$ and hence $c=\frac{2k}{N}-1$. 
\end{proof}
The first case described in the Proposition above corresponds to an incoherent state, where the wave functions are "symmetrically" positioned with respect to the zero wave function in $L^2(\R^d)$. In this case the wave functions evolve independently according to $\psi_j(t)=e^{-itH}\psi_{j, 0}$, where the Hamiltonian is given by $H=-\frac12\Delta+V$.
The second case corresponds to synchronization; namely $k=N$ will be the complete phase synchronization, while $1\leq k\leq N-1$ will provide frequency synchronization only.

In this context, we are going to investigate  under which assumptions the system \eqref{eq:SL2} exhibits complete phase synchronization. As before we will denote by  $\tilde r_j=\RE\langle\zeta, \psi_j\rangle$, $\tilde s_j=\IM\langle\zeta, \psi_j\rangle$ the macroscopic correlations of the wave functions with the order parameter.
\begin{proposition}\label{prop:phase_synch}
Let $(\psi_1, \dotsc, \psi_N)$ be solutions to \eqref{eq:SL2}, with $(\psi_1(0), \dotsc, \psi_N(0))=(\psi_{1, 0}, \dotsc, \psi_{N, 0})$ and with $\Omega_j=0$ for any $j=1,\dotsc, N$. Let us furthermore assume 
$\tilde r_j(0)>0$ for any $j=1, \dotsc, N$, where $\tilde r_j$ is defined in \eqref{eq:macro_corr}. Then we have
\begin{equation}\label{eq:bdr}
|1-\tilde r_j(t)|^2+|\tilde s_j(t)|^2\lesssim e^{-Kt},\quad\textrm{as}\;t\to\infty.
\end{equation}
\end{proposition}
\begin{proof}
The statement of this Proposition is equivalent to prove that the quantities $1-\langle\zeta, \psi_j\rangle$ exponentially converge to zero for large times, for any $j=1, \dotsc, N$.
For this reason let us write \eqref{eq:corr} in the following way
\begin{equation*}
\begin{aligned}
\frac{d}{dt}\left(1-\langle\zeta, \psi_j\rangle\right)=&i\Omega_j-i\Omega_j(1-\langle\zeta, \psi_j\rangle)+\frac{i}{N}\sum_{\ell=1}^N\Omega_\ell(1-\langle\psi_\ell, \psi_j\rangle)\\
&-\frac{K}{2}\left[2(1-\langle\zeta, \psi_j\rangle)-(1-\langle\zeta, \psi_j\rangle)^2-\frac1N\sum_{\ell=1}^N(1-\langle\psi_\ell, \zeta\rangle)(1-\langle\psi_\ell, \psi_j\rangle)\right].
\end{aligned}
\end{equation*}
By setting $\Omega_j=0$ for all $j=1, \dotsc, N$, we have
\begin{equation*}
\frac{d}{dt}\left(1-\langle\zeta, \psi_j\rangle\right)=-\frac{K}{2}\left[2(1-\langle\zeta, \psi_j)-(1-\langle\zeta, \psi_j\rangle)^2-\frac1N\sum_{\ell=1}^N(1-\langle\psi_\ell, \zeta\rangle)(1-\langle\psi_\ell, \psi_j\rangle)\right].
\end{equation*}
For simplicity, let us define 
\begin{equation*}
1-\langle\zeta, \psi_j\rangle=:f_j+ig_j,
\end{equation*}
so that $f_j=1-\tilde r_j$, $g_j=-\tilde s_j$, and analogously $1-\langle\psi_\ell, \psi_j\rangle=:f_{\ell j}+ig_{\ell j}$, with $f_{\ell j}, g_{\ell j}$ real-valued functions. By separating the real and the imaginary parts in the previous ODE, we have
\begin{equation}\label{eq:corr_r}
\left\{\begin{aligned}
\frac{d}{dt}f_j=&-\frac{K}{2}\left[2f_j-(f_j^2-g_j^2)-\frac1N\sum_{\ell=1}^N(f_{\ell}f_{\ell j}+g_\ell g_{\ell j})\right]\\
\frac{d}{dt}g_j=&-\frac{K}{2}\left[2g_j-2f_jg_j-\frac1N\sum_{\ell=1}^N(f_\ell g_{\ell j}-g_\ell f_{\ell j})\right].
\end{aligned}\right.
\end{equation}
To investigate the  system \eqref{eq:corr_r}, we consider the quantity 
\begin{equation*}
\frac{1}{2N}\sum_{j=1}^N\left(f_j^2+g_j^2\right),
\end{equation*}
then by differentiating in time and by using that $g_{\ell j}=-g_{j\ell}$, we have
\begin{equation}\label{eq:Lyap}
\frac{d}{dt}\left(\frac{1}{2N}\sum_{j=1}^n(f_j^2+g_j^2)\right)=-\frac{K}{2N}\sum_{j=1}^N(2-f_j)(f_j^2+g_j^2)
+\frac{K}{2N^2}\sum_{j, \ell=1}^Nf_{\ell j}(f_jf_\ell+g_jg_\ell).
\end{equation}
Since  $f_{\ell j}\geq 0$, hence by the Young's inequality, we have
\begin{equation*}
\frac{K}{2N^2}\sum_{j, \ell=1}^Nf_{\ell j}(f_jf_\ell-g_jg_\ell)\leq\frac{K}{2N}\sum_{j=1}^Nf_{j}\left(f_j^2+g_j^2\right),
\end{equation*}
where we used that $\frac1N\sum_\ell f_{j\ell}=\frac1N\sum_\ell f_{\ell j}=f_j$. As a consequence we obtain
\begin{equation*}
\frac{d}{dt}\left(\frac{1}{2N}\sum_{j=1}^N(f_j^2+g_j^2)\right)\leq-\frac{K}{2N}\sum_{j=1}^N(1-f_j)(f_j^2+g_j^2).
\end{equation*}
Moreover since  $1-f_j(0)>0$, for any $j=1, \dotsc, N$, then  $\frac{1}{2N}\sum_{j=1}^N(f_j^2+g_j^2)$ is exponentially decreasing in time. Consequently we also have $1-f_j(t)>1-f_j(0)$. for all positive times and then, for sufficiently large times, we have
\begin{equation*}
\frac{1}{2N}\sum_{j=1}^N\left(f_j^2+g_j^2\right)(t)\leq e^{-Kt}\frac{1}{2N}\sum_{j=1}^N(f_j(0)^2+g_j(0)^2).
\end{equation*}
\end{proof}

The previous Proposition yields to
\begin{equation*}
\lim_{t\to\infty}\langle\zeta, \psi_j\rangle(t)=1,\quad\textrm{for any}\;j=1, \dotsc, N,
\end{equation*}
with an explicit decay rate given by $|1-\langle\zeta, \psi_j(t)\rangle|\lesssim e^{-Kt}$, as $t\to\infty$. 

As a consequence of this fact, we get indeed the complete phase synchronization for the system \eqref{eq:SL2}. Moreover we will show that, once we prove complete synchronization, it will be also possible to further improve the decay rate in \eqref{eq:bdr}. \\
We notice that \eqref{eq:gdr} below provides indeed the same decay rate obtained for the two oscillators case, when $\Omega=0$, see \eqref{eq:gdr_two}.
\begin{proposition}\label{prop:sync_id}
Under the same hypotheses of Proposition \ref{prop:phase_synch} we have
\begin{equation*}
\lim_{t\to\infty}\|\psi_j(t)-\psi_k(t)\|_{L^2}=0,\qquad\forall\;j, k=1, \dotsc, N
\end{equation*}
and moreover
\begin{equation}\label{eq:gdr}
\|\psi_j(t)-\psi_k(t)\|_{L^2}\lesssim e^{-Kt},\quad\textrm{as}\;t\to\infty.
\end{equation}
\end{proposition}
\begin{proof}
From system \eqref{eq:SL2}, we get the system of ODEs satisfied by $\langle\psi_j, \psi_k\rangle$, namely
\begin{equation*}
\frac{d}{dt}\langle\psi_j, \psi_k\rangle=\frac{K}{2}\left(\langle\psi_j, \zeta\rangle+\langle\zeta, \psi_k\rangle\right)\left(1-\langle\psi_j, \psi_k\rangle\right),
\end{equation*}
or analogously, by using the same notations as in the Proposition \ref{prop:phase_synch},
\begin{equation*}
\left\{\begin{aligned}
\dot f_{jk}=&-\frac{K}{2}(2-f_j-f_k)f_{jk}+\frac{K}{2}(g_j-g_k)g_{jk}\\
\dot g_{jk}=&-\frac{K}{2}(2-f_j-f_k)g_{jk}-\frac{K}{2}(g_j-g_k)f_{jk}.
\end{aligned}\right.
\end{equation*}
Let us now consider the quantity $\frac12(f_{jk}^2+g_{jk}^2)$, by differentiating it with respect to time we obtain
\begin{equation*}
\frac{d}{dt}\left(\frac12(f_{jk}^2+g_{jk}^2)\right)=-\frac{K}{2}\left(2-f_j-f_k\right)(f_{jk}^2+g_{jk}^2)
\end{equation*}
and by integrating the above ODE, we get
\begin{equation*}
\frac12\left(f_{jk}^2+g_{jk}^2\right)=e^{-K\int_0^t(2-f_j-f_k)(s)\,ds}\left(f_{jk}^2+g_{jk}^2\right)(0).
\end{equation*}
By using the exponential convergence to 1 of $f_j, f_k$ proved in the previous Proposition, we then infer
\begin{equation*}
|1-r_{jk}(t)|^2+|s_{jk}(t)|^2\lesssim e^{-Kt}.
\end{equation*}
Let us notice that in the last equality we don't get the optimal decay rate. However, we can now improve it by going back to the inequality \eqref{eq:Lyap}. By using the results obtained so far, it follows 
\begin{equation*}
\begin{aligned}
\frac{1}{2N^2}\sum_{j=1}^N\left(f_j^2+g_j^2\right)(t)\leq&e^{-2Kt}\frac{1}{2N^2}\sum_{j=1}^N\left(f_j^2(0)+g_j^2(0)\right)\\
&+\int_0^te^{-2K(t-s)}\bigg(\frac{K}{2N}\sum_{j=1}^N(f_j^3(s)+f_j(s)g_j^2(s))\\
&\quad+\frac{K}{2N^2}\sum_{j, \ell=1}^Nf_{\ell j}(s)(f_j(s)f_{\ell}(s)+g_j(s)g_\ell(s))\bigg)\,ds.
\end{aligned}
\end{equation*}
By repeating the same argument as before, one has
\begin{equation*}
|1-r_{jk}(t)|^2+|s_{jk}(t)|^2\lesssim e^{-2Kt}.
\end{equation*}
Since $r_{jk}(t)=\RE\langle\psi_j, \psi_k\rangle(t)$, we get
\begin{equation*}
\|\psi_j(t)-\psi_k(t)\|_{L^2}^2\lesssim 2(1-r_{jk}(t))\lesssim e^{-Kt}.
\end{equation*}
\end{proof}
Let us now show that, under the same hypotheses as Proposition \ref{prop:phase_synch}, we can also prove phase synchronization in $H^1(\R^d)$.
\begin{theorem}\label{thm:H1_N}
Let $(\psi_1, \dotsc, \psi_N)$ be solutions to \eqref{eq:SL2},  with $\psi_1(0), \dotsc, \psi_N(0))=(\psi_{1, 0}, \dotsc, \psi_{N, 0})\in H^1(\R^d)$ and $\Omega_j=0$ for any $j=1, \dotsc, N$. Let us furthermore assume 
$\tilde r_j(0)>0$, for any $j=1, \dotsc, N$. Then we have
\begin{equation*}
\lim_{t\to\infty}\|\psi_j(t)-\psi_k(t)\|_{H^1}=0,\quad\textrm{as}\;j,k=1,\dotsc, N.
\end{equation*}
\end{theorem}
\begin{proof}
We proceed as for the two oscillators case. Let us first define the total energy
\begin{equation*}
E(t)=\frac1N\sum_{j=1}^NE_j(t), \quad E_j(t)=\int\frac12|\nabla\psi_j|^2+V|\psi_j|^2\,dx
\end{equation*}
and let us show that it remains uniformly bounded for all times. Indeed, by differentiating it with respect to time, we obtain
\begin{equation*}
\begin{aligned}
\frac{d}{dt}E(t)=&\frac1N\sum_{j=1}^N2\int\RE\left\{\left(-\frac12\Delta\bar\psi_j+V\bar\psi_j\right)\frac{K}{2}(\zeta-\langle\zeta, \psi_j\rangle\psi_j)\right\}\,dx\\
=&-\frac{K}{N}\sum_{j=1}^N\tilde r_j(t)E_j(t)+\frac{K}{N}\sum_{j=1}^N\int\RE\left\{\frac12\nabla\bar\psi_j\cdot\nabla\zeta+V\bar\psi_j\zeta\right\}\,dx.
\end{aligned}
\end{equation*}
We write the last sum in the following way
\begin{equation*}
\frac1N\sum_{j=1}^N\int\RE\left\{\frac12\nabla\bar\psi_j\cdot\nabla\zeta+V\bar\psi_j\zeta\right\}\,dx
=\frac{1}{N^2}\sum_{j, k=1}^N\RE\left\{\frac12\nabla\bar\psi_j\cdot\nabla\psi_k+V\bar\psi_j\psi_k\right\}\,dx
\end{equation*}
and we notice that
\begin{equation}\label{eq:102}
\int\RE\left\{\frac12\nabla\bar\psi_j\cdot\nabla\psi_k+V\bar\psi_j\psi_k\right\}\,dx=-\frac12E_{jk}(t)+\frac12\left(E_j(t)+E_k(t)\right),
\end{equation}
where we defined
\begin{equation*}
E_{jk}(t)=\int\frac12|\nabla(\psi_j-\psi_k)|^2+V|\psi_j-\psi_k|^2\,dx.
\end{equation*}
By putting everything together we obtain
\begin{equation*}\begin{aligned}
\frac{d}{dt}E(t)=&\frac{K}{N}\sum_{j=1}^N(1-\tilde r_j(t))E_j(t)-\frac{K}{2N^2}\sum_{j, k=1}^NE_{jk}(t)\\
\leq&\frac{K}{N}\sum_{j=1}^N(1-\tilde r_j(t))E_j(t).
\end{aligned}\end{equation*}
Since we already proved that $|1-\tilde r_j(t)|\lesssim e^{-Kt}$, as $t\to\infty$, then by using Gronwall's inequality it follows
\begin{equation*}
E(t)\leq CE(0).
\end{equation*}
To obtain our synchronization results  we are going to use now the relative energy
\begin{equation*}
\tilde E(t):=\frac{1}{2N^2}\sum_{j, k=1}^NE_{jk}(t).
\end{equation*}
One has
\begin{equation}\label{eq:101}
E(t)=E_z(t)+\tilde E(t),
\end{equation}
where 
\begin{equation*}
E_z(t):=\int\frac12|\nabla\zeta|^2+V|\zeta|^2\,dx,
\end{equation*}
since
\begin{equation*}
\begin{aligned}
E_z(t)=&\frac{1}{N^2}\sum_{j, k=1}^N\int\RE\left\{\frac12\nabla\bar\psi_j\cdot\nabla\psi_k+V\bar\psi_j\psi_k\right\}\,dx\\
=&-\frac{1}{2N^2}\sum_{j, k=1}^NE_{jk}(t)+\frac{1}{N}\sum_{j=1}^NE_j(t),
\end{aligned}
\end{equation*}
where we used again formula \eqref{eq:102}.
By using \eqref{eq:101}, we then have
\begin{equation*}
\begin{aligned}
\frac{d}{dt}\tilde E(t)=&\frac{d}{dt}\left(E(t)-E_z(t)\right)\\
=&\frac{K}{N}\sum_{j=1}^N(1-\tilde r_j(t))E_j(t)-K\tilde E(t)\\
&-2\int\RE\left\{\left(-\frac12\Delta\bar\zeta+V\bar\zeta\right)\frac{K}{2}\left(\zeta-\frac1N\sum_{\ell=1}^N\langle\zeta, \psi_\ell\rangle\psi_\ell\right)\right\}\,dx\\
=&\frac{K}{N}\sum_{j=1}^N(1-\tilde r_j(t))E_j(t)-K\tilde E(t)-KE_z(t)\\
&+\frac{K}{N}\sum_{\ell=1}^N\RE\left\{\langle\zeta, \psi_\ell\rangle\int\frac12\nabla\bar\zeta\cdot\nabla\psi_\ell+V\bar\zeta\psi_\ell\,dx\right\}.
\end{aligned}\end{equation*}
Let us consider the last term on  the right hand side, it can be written as
\begin{multline*}
\frac{K}{N^2}\sum_{j, \ell=1}^N\RE\left\{\langle\psi_j, \psi_\ell\rangle\int\frac12\nabla\bar\psi_j\cdot\nabla\psi_\ell+V\bar\psi_j\psi_\ell\,dx\right\}\\
=\frac{K}{N^2}\sum_{j, \ell=1}^Nr_{j\ell}\int\RE\left\{\frac12\nabla\bar\psi_j\cdot\nabla\psi_\ell+V\bar\psi_j\psi_\ell\right\}\,dx
-\frac{K}{N^2}\sum_{j, \ell=1}^Ns_{j\ell}\int\IM\left\{\frac12\nabla\bar\psi_j\cdot\nabla\psi_\ell+V\bar\psi_j\psi_\ell\right\}\,dx.
\end{multline*}
The first summand can be written as 
\begin{equation*}
\frac{K}{N^2}\sum_{j, \ell=1}^Nr_{j\ell}\left[-\frac12E_{jk}(t)+\frac12(E_j(t)+E_\ell(t))\right]
=-\frac{K}{2N^2}\sum_{j, \ell=1}^Nr_{j\ell}E_{j\ell}+\frac{K}{N}\sum_{j=1}^N\tilde r_jE_j,
\end{equation*}
where we used the symmetry $r_{j\ell}=r_{\ell j}$ and the fact that $\tilde r_j=\frac1N\sum r_{j\ell}$. By putting everything together we then obtain
\begin{equation*}\begin{aligned}
\frac{d}{dt}\tilde E(t)=&KE(t)-\frac{K}{2N^2}\sum_{j, k=1}^N(1+r_{jk}(t))E_{jk}(t)-KE_z(t)\\
&-\frac{K}{N^2}\sum_{j, k=1}^Ns_{jk}(t)\int\IM\left\{\frac12\nabla\bar\psi_j\cdot\nabla\psi_k+V\bar\psi_j\psi_k\right\}\,dx\\
=&-\frac{K}{N}\sum_{j, k=1}^Nr_{jk}(t)E_{jk}(t)-\frac{K}{N^2}\sum_{j, k=1}^Ns_{jk}(t)\int\IM\left\{\frac12\nabla\bar\psi_j\cdot\nabla\psi_k+V\bar\psi_j\psi_k\right\}\,dx.
\end{aligned}\end{equation*}
We can write the last equality as
\begin{equation*}\begin{aligned}
\frac{d}{dt}\tilde E(t)=&-K\tilde E(t)+\frac{K}{2N^2}\sum_{j, k=1}^N(1-r_{jk}(t))E_{jk}(t)\\
&-\frac{K}{N^2}\sum_{j, k=1}^Ns_{jk}(t)\int\IM\left\{\frac12\nabla\bar\psi_j\cdot\nabla\psi_k+V\bar\psi_j\psi_k\right\}\,dx\\
\leq&-K\tilde E(t)+\frac{K}{2N^2}\sum_{j, k=1}^N(1-r_{jk}(t))E_{jk}(t)+\frac{K}{2N^2}\sum_{j, k=1}^N|s_{jk}(t)|E_j(t).
\end{aligned}\end{equation*}
Therefore by using the Gronwall's inequality and by exploiting the optimal decay rate for $1-r_{jk}(t)$ and $s_{jk}(t)$ we obtain
\begin{equation*}
\tilde E(t)\leq e^{-Kt}\tilde E(0)+\frac{K}{2N^2}\sum_{j, k=1}^N\int_0^te^{-K(t-s)}e^{-Ks}\left(E_{jk}(s)+E_j(s)\right)\,ds.
\end{equation*}
By noticing that both $E_{jk}$ and $E_j$ are uniformly bounded in time we then obtain
\begin{equation*}
\tilde E(t)\leq e^{-KT}\left(\tilde E(0)+CtE(0)\right).
\end{equation*}
\end{proof}
\section{Alignment of the Schr\"odinger momenta}
We devote our last Section in showing that synchronization holds also for the momenta associated to the wave functions $\psi_j$, namely the probability densisties $\rho_j:=|\psi_j|^2$, and the current densities 
$J_j:=\IM(\bar\psi_j\nabla\psi_j)$.

Having at hand the $H^1$ synchronization results proved in the previous Sections it is now straightforward to show that, under the same assumptions of Theorems \ref{prop:H1_sync} and \ref{thm:H1_N}, we have
\begin{equation*}
\lim_{t\to\infty}\|\rho_j(t)-\rho_k(t)\|_{L^1}=\lim_{t\to\infty}\|J_j(t)-J_k(t)\|_{L^1}=0, \quad\textrm{for any}\;j, k=1, \dotsc, N.
\end{equation*}
Let us remark that, even in the case of non-identical oscillators, that is in the case when we only have complete phase synchronization for the wave functions, the momenta exhibit anyway complete frequency localization, or more precisely they show alignment.
\section*{Acknowledgement}
After we obtained the results in the present paper, which were anticipated by P.Marcati on March 2016 in \cite{Pad}, we learned that similar results were obtained independently by Ha and Huh in \cite{HH}.

\end{document}